\documentclass[10pt]{amsart}
\usepackage{amssymb,amsmath,txfonts,amsthm}
\usepackage{hyperref}
\usepackage{mathrsfs}
\newtheorem{theorem}{Theorem}

\newtheorem{lemma}{Lemma}
\newtheorem*{remark}{Remark}
\newtheorem{claim}{Claim}
\newtheorem*{definition}{Definition}
\newtheorem{cor}{Corollary}
\newtheorem*{conj}{Conjecture}

\def\Xint#1{\mathchoice
  {\XXint\displaystyle\textstyle{#1}}%
  {\XXint\textstyle\scriptstyle{#1}}%
  {\XXint\scriptstyle\scriptscriptstyle{#1}}%
  {\XXint\scriptscriptstyle\scriptscriptstyle{#1}}%
  \!\int}
\def\XXint#1#2#3{{\setbox0=\hbox{$#1{#2#3}{\int}$}
  \vcenter{\hbox{$#2#3$}}\kern-.5\wd0}}

\def\dashint{\Xint-}

\author{Gang Liu}
\address{Department of Mathematics\\University of California, Berkeley\\Berkeley, CA 94720}
\email{gangliu@math.berkeley.edu}
\title[Maximal Volume Growth]{On the volume growth of K\"ahler manifolds with nonnegative bisectional curvature}
\date{}
\begin{document}
\begin{abstract}
Let $M$ be a complete K\"ahler manifold with nonnegative bisectional curvature. Suppose the universal cover does not split and $M$ admits a nonconstant holomorphic function with polynomial growth, we prove $M$ must be of maximal volume growth.
This confirms a conjecture of Ni in \cite{[N1]}. There are two essential ingredients in the proof: the Cheeger-Colding theory \cite{[CC1]}-\cite{[CC4]} on Gromov-Hausdorff convergence of manifolds; the three circle theorem for holomorphic functions in \cite{[L]}.
\end{abstract}
\maketitle

\section{\bf{Introduction}}
In \cite{[Y1]}, Yau proposed the study of the uniformization of complete K\"ahler manifolds with nonnegative bisectional curvature. In particular, one wishes to determine whether or not a complete K\"ahler manifold with positive bisectional curvature is biholomorphic to $\mathbb{C}^n$. Motivated by this, Yau further asked whether or not the ring of holomorphic functions with polynomial growth is finitely generated, and whether or not the dimension of the
spaces of holomorphic functions of polynomial growth is bounded from above by the
dimension of the corresponding spaces of polynomials on $\mathbb{C}^n$.

In \cite{[N1]}, Ni confirmed Yau's conjecture on the sharp dimension estimate of holomorphic functions with polynomial growth when the manifold has maximal volume growth. Here maximal volume growth means $$\frac{Vol(B(p, r))}{r^{2n}} >c>0$$ for all $r>0$, $p\in M$.
\begin{definition}\label{def1}
Let $M$ be a complete noncompact K\"ahler manifold. Let $\mathcal{O}(M)$ be the ring of holomorphic functions on $M$. For any $d\geq 0$, define $$\mathcal{O}_d(M) = \{f\in \mathcal{O}(M)|\overline{\lim\limits_{r\to\infty}}\frac{M_f(r)}{r^d}<\infty\}.$$ Here $r$ is the distance from a fixed point $p$ on $M$; $M_f(r)$ is the maximal modulus of $f$ on $B(p, r)$. If $f\in \mathcal{O}_d(M)$, we say $f$ is of polynomial growth with order $d$. Let $$\mathcal{O}_P(M) = \cup_{d\in\mathbb{N}}\mathcal{O}_d(M).$$ If $M$ is only Riemannian, let $\mathcal{H}_d(M)$ be the linear space of harmonic functions on $M$ of polynomial growth with order $d$.
 \end{definition}\begin{theorem}\label{thm-9}[Ni]
Let $M^n$ be a complete K\"ahler manifold with nonnegative holomorphic bisectional curvature. Assume $M$ is of maximal volume growth, then $$dim(\mathcal{O}_d(M))\leq dim(\mathcal{O}_d(\mathbb{C}^n))$$ for any positive integer $d$. If the equality holds for some $d$, $M$ is isometric and biholomorphic to $\mathbb{C}^n$.
\end{theorem}
Later Chen, Fu, Le, Zhu \cite{[CFLZ]} removed the maximal volume growth condition by using the same technique in \cite{[N1]}. See also \cite{[L]} for a different proof. Based on some results in \cite{[NT1]} and \cite{[N1]}, 
Ni raised the following conjecture (Conjecture $3.1$ in \cite{[N1]}):
\begin{conj}
Let $M^n$ be a complete noncompact K\"ahler manifold with nonnegative bisectional curvature. Assume $M$ admits a nonconstant holomorphic function with polynomial growth and the bisectional curvature is positive at least at one point, then $M$ is of maximal volume growth. Namely, $\mathcal{O}_P(M)\neq \mathbb{C}$, average quadratic curvature decay, and $M$ being of maximal volume growth are all equivalent if $M$ has quasi-positive bisectional curvature. Average quadratic curvature decay means for all $r>0$,
\begin{equation}\label{eq1}
\dashint_{B(p, r)} S \leq \frac{C}{r^2}
\end{equation} where $p$ is a point on $M$, $C$ is a positive constant independent of $r$, $S$ is the scalar curvature. $\dashint$ means the average.
\end{conj}

In complex one dimensional case, the conjecture is well-known, e.g., \cite{[LT]}. In general dimensions, 
under the assumption of the conjecture, Ni proved that $Vol(B(p, r))\geq cr^{n+1}$ in \cite{[N1]}. Under an extra pointwise curvature decay condition, Ni and Tam \cite{[NT1]} were able to confirm the conjecture.  Proofs of the partial results in \cite{[N1]} and \cite{[NT1]} are based on the Poincare-Lelong equation, heat flow methods including the sharp mononoticity formula discovered in \cite{[N1]}. Very recently, in a personal conversation with Ni, the author was informed that the conjecture is known to be true if one assumes an upper bound of curvature. The proof involves the K\"ahler-Ricci flow. In this note, we confirm the first part of Ni's conjecture:
\begin{theorem}\label{thm1}
Let $M^n$ be a complete noncompact K\"ahler manifold with nonnegative bisectional curvature. Suppose the universal cover $\tilde{M}$ does not split as a product of two K\"ahler manifolds. If $M$ admits a nonconstant holomorphic function with polynomial growth, then $M$ has maximal volume growth.
\end{theorem}
\begin{remark}
Obviously it is necessary to assume $\tilde{M}$ does not split. This theorem essentially reduces Yau's conjecture on the finite generation of holomorphic functions with polynomial growth to the case when the manifold has maximal volume growth.
\end{remark}

\begin{cor}\label{cor1}
Let $M^n$ be  a complete K\"ahler manifold with nonnegative bisectional curvature. Suppose there exists a positive constant $c>0$ so that  $dim(\mathcal{O}_d(M))\geq cd^n$ for some sufficiently large $d$, then $M$ is of maximal volume growth. \end{cor}
\begin{remark}
Corollary \ref{cor1} holds under the weaker assumption that $M$ has nonnegative Ricci curvature and nonnegative holomorphic sectional curvature. At this moment, it is unclear to the author whether we still have corollary \ref{cor1} if we only assume the Ricci curvature to be nonnegative.
\end{remark}

It is interesting to compare the corollary with a theorem of Colding and Minicozzi \cite{[CM6]} (corollary $6.51$):
\begin{theorem}[Colding-Minicozzi]
Let $M^m$ be a complete noncompact Riemannian manifold with nonnegative Ricci curvature. Suppose there exists a positive constant $c>0$ so that $dim(\mathcal{H}_d(M))\geq cd^{m-1}$ for some sufficiently large $d$, then $M$ is of maximal volume growth.
\end{theorem}

\begin{cor}\label{cor2}
Let $M^n$ be a complete noncompact K\"ahler manifold with nonnegative bisectional curvature. Let $p\in M$. Suppose the Ricci curvature is positive at one point and the scalar curvature has average quadratic decay, i.e., (\ref{eq1}) holds. Then $M$ is of maximal volume growth.
\end{cor} 
\begin{remark}
One cannot remove the condition that $M$ has Ricci positive at one point. For instance, $M$ could have a flat torus factor.
\end{remark}

The proof of theorem \ref{thm1} is different from the arguments of Ni \cite{[N1]} and Ni-Tam \cite{[NT1]}. In our proof, theory on the Gromov-Hausdorff convergence \cite{[CC1]}-\cite{[CC4]} and the three circle theorem \cite{[L]} are crucial. We argue by contradiction. First blow down the manifold, then blow up at some regular point to get a real Euclidean space whose dimension is strictly smaller than the dimension of $M$. Then by three circle theorem, we can pass all holomorphic functions with polynomial growth to that Euclidean space. Finally the contradiction comes from dimension estimate: The dimension of the Euclidean space is too low while the dimension of functions is too high.
\begin{remark}
The statement of theorem \ref{thm1} is purely on smooth manifolds.  However, in our proof, we have to turn to some possibly singular collapsed limit. The final contradiction comes
from the tangent cone of the singular collapsed space. Thus the proof of theorem \ref{thm1} might be of some independent interest.
\end{remark}

\begin{center}
\bf  {\quad Acknowledgment}
\end{center}
The author would like to express his deep gratitude to Professors John Lott, Lei Ni, Jiaping Wang for many valuable discussions during the work. He also thanks Professor Luen-Fai Tam for the interest in this work. 

\section{\bf{Preliminary results}}
In this section, we collect some results required in the proof of theorem \ref{thm1}.

First recall some convergence results for manifolds with Ricci curvature lower bound. 
Let $(M^n_i, y_i, \rho_i)$ be a sequence of pointed complete Riemannian manifolds, where $y_i\in M^n_i$ and $\rho_i$ is the metric on $M^n_i$. By Gromov's compactness theorem, if $(M^n_i, y_i, \rho_i)$ have a uniform lower bound of the Ricci curvature, then a subsequence converges to some $(M_\infty, y_\infty, \rho_\infty)$ in the Gromov-Hausdorff topology. See \cite{[G]} for the definition and basic properties of Gromov-Hausdorff convergence.
\begin{definition}
Let $K_i\subset M^n_i\to K_\infty\subset M_\infty$ in the Gromov-Hausdorff topology. Assume $\{f_i\}_{i=1}^\infty$ are functions on $M^n_i$, $f_\infty$ is a function on $M_\infty$.  
$\Phi_i$ are $\epsilon_i$-Gromov-Hausdorff approximations, $\lim\limits_{i\to\infty} \epsilon_i = 0$. If $f_i\circ \Phi_i$ converges to $f_\infty$ uniformly, we say $f_i\to f_\infty$ uniformly over $K_i\to K_\infty$.
\end{definition}
 In many applications, $f_i$ are equicontinuous. The Arzela-Ascoli theorem applies to the case when the spaces are different.  When $(M_i^n, y_i, \rho_i)\to (M_\infty, y_\infty, \rho_\infty)$ in the Gromov-Hausdorff topology, any bounded, equicontinuous sequence of functions $f_i$ has a subsequence converging uniformly to some $f_\infty$ on $M_\infty$.

As in section $9$ of \cite{[Che]}, we have the following definition.
\begin{definition}
If $\nu_i, \nu_\infty$ are Borel regular measures on $M_i^n, M_\infty$, we say that $(M^n_i, y_i, \rho_i, \nu_i)$ converges  to $(M_\infty, y_\infty, \rho_\infty, \nu_\infty)$ in the measured Gromov-Hausdorff sense, if $(M^n_i, y_i, \rho_i, \nu_i)\to (M_\infty, y_\infty, \rho_\infty, \nu_\infty)$ in the Gromov-Hausdorff topology and for any $x_i\to x_\infty$ ($x_i\in M_i, x_\infty\in M_\infty$), $r>0$, $\nu_i(B(x_i, r))\to \nu_\infty(B(x_\infty, r))$.
\end{definition}
For any sequence of manifolds with Ricci curvature lower bound, after suitable renormalization of the volume, there is a subsequence converging in the measured Gromov-Hausdorff sense. If the volume is noncollapsed, $\nu_\infty$ is just the $n$-dimensional Hausdorff measure of $M_\infty$. See \cite{[CC2]}.

Let the complete pointed metric space $(M_\infty^m, y)$ be the Gromov-Hausdorff limit of a sequence of connected pointed Riemannian manifolds, $\{(M_i^n, p_i)\}$, with $Ric(M_i)\geq 0$. Here $M_\infty^m$ has Haudorff dimension $m$ with $m\leq n$. A tangent cone at $y\in M_\infty^m$ is a complete pointed Gromov-Hausdorff limit $((M_\infty)_y, d_\infty, y_\infty)$ of $\{(M_\infty, r_i^{-1}d, y)\}$, where $d, d_\infty$ are the metrics of $M_\infty, (M_\infty)_y$ respectively, $\{r_i\}$ is a positive sequence converging to $0$.

\begin{definition}
A point $y\in M_\infty$ is called regular, if there exists some $k$ so that every tangent cone at $y$ is isometric to $\mathbb{R}^k$. A point is called singular, if it is not regular.
\end{definition}
Now we introduce a theorem of Cheeger and Colding (theorem $2.1$ in \cite{[CC2]}) which is crucial in the proof of theorem \ref{thm1}.
\begin{theorem}[Cheeger-Colding]
For any renormalized limit measure, the singular set has measure $0$. In particular, the regular points are dense.
\end{theorem}
\begin{remark}
There is a typo on page $420$ in \cite{[CC2]}: $Y^m$ should be the limit of $M_i^n$, not $M_i^m$. Since this is the arbitrary Gromov-Hausdorff convergence, the dimension might decrease. See also paragraph $4$ on page $409$ in \cite{[CC2]}.
\end{remark}

For a Lipschitz function $f$ on $M_\infty$, define a norm $||f||^2_{1, 2} = ||f||^2_{L^2}+\int_{M_\infty}|Lip f|^2$, where $$Lip(f, x) =\lim\sup\limits_{y\to x}\frac{|f(y)-f(x)|}{d(x, y)}.$$
In \cite{[Che]}, a Sobolev space $H_{1, 2}$ is defined by taking the closure of the norm $||\cdot||_{1, 2}$ for Lipschitz functions.

\emph{Condition (1)}:
$M_\infty$ satisfies the volume doubling property if for any $r>0$, $x\in M_\infty$,  $\nu_\infty(B(x, 2r))\leq 2^n\nu_\infty(B(x, r))$.

\emph{Condition (2)}:
$M_\infty$ satisfies the weak Poincare inequality if $$\int_{B(x, r)}|f - \overline{f}|^2 \leq C(n)r^2\int_{B(x, 2r)}|Lip f|^2$$ for all Lipschitz functions.
Here $\overline f$ is the average of $f$ on $B(x, r)$.

In theorem $6.7$ of \cite{[CC4]}, it was proved that if $M_\infty$ satisfies the $\nu$-rectifiability condition, condition (1) and condition (2), then there is a unique differential $df$ for $f\in H_{1, 2}$.
If $f$ is Lipschitz, $\int |Lipf|^2 = \int |df|^2$. Moreover, the $H_{1, 2}$ norm becomes an inner product. Therefore $H_{1, 2}$ is a Hilbert space. Then there exists a unique self-adjoint operator $\Delta$ on $M_\infty$ such that $$\int_{M_\infty} <df, dg> = \int _{M_\infty}<\Delta f, g>$$ for all Lipschitz functions on $M_\infty$ with compact support (Of course we can extend the functions to Sobolev spaces). See theorem $6.25$ of \cite{[CC4]}.

If $M_i\to M_{\infty}$ in the measured Gromov-Hausdorff sense and that the Ricci curvature is nonnegative for all $M_i$, then the $\nu$-rectifiability of $M_\infty$ was proved in theorem $5.5$ in \cite{[CC4]}.
By the volume comparison, Condition (1) obviously holds for $M_\infty$. Condition (2) also holds. See \cite{[X]} for a proof.

In \cite{[Di1]}\cite{[X]}, the following lemma was proved:
\begin{lemma}\label{lemma-10}
Suppose $M_i$ has nonnegative Ricci curvature and $M_i\to M_\infty$ in the measured Gromov-Hausdorff sense. Let $f_i$ be Lipschitz functions on $B(x_i, 2r)\subset M_i$ satisfying $\Delta f_i = 0$; $|f_i|\leq L, |\nabla f_i|\leq L$ for some constant $L$. Assume $x_i\to x_\infty$, $f_i\to f_\infty$ on $M_\infty$. Then $\Delta f_\infty = 0$ on $B(x_\infty, r)$.
\end{lemma}

\bigskip

Next we introduce the following theorem which is corollary $1$  in \cite{[L]}. This will be another key ingredient in the proof of theorem \ref{thm1}.
\begin{theorem}\label{thm0}
Let $M$ be a complete K\"ahler manifold with nonnegative holomorphic sectional curvature, $p\in M$. For a holomorphic function $f$ on $M$, let $M(r)=\max |f(x)|$ for $x\in B(p, r)$. Then $f\in \mathcal{O}_d(M)$ if and only if $\frac{M(r)}{r^d}$ is nonincreasing.\end{theorem}

\section{\bf{Proof of theorem \ref{thm1}}}

\emph{Proof of theorem \ref{thm1}:}
Assume $M^n$ is not of maximal volume growth. Fix a point $p\in M$, consider the rescaled sequence of manifolds $(M'_i, p_i, g'_i) = (M, p, r_i^{-2}g)$ where $r_i$ is a sequence tending to $\infty$. Then by Gromov's compactness theorem, we may assume $(M'_i, p_i, g'_i)\to (N, p_\infty, g_\infty)$ in the measured Gromov-Hausdorff sense where $N$ is a metric measured space. By our assumption and theorem $3.1$ in \cite{[CC2]}, $N$ has Hausdorff dimension less than or equal to $2n-1$. Now theorem 2.1 in \cite{[CC2]} implies that the regular points for $N$ are dense.  Therefore we can find a point $q\in N$ where the tangent cone is isometric to $\mathbb{R}^k$. Here $k\leq 2n-1$. This means that for any $\epsilon>0$, $R>0$, we can find a fixed $r>0$ so that the metric ball $(B_{g_\infty}(q, rR), \frac{1}{r^2}g_\infty)$ is $\epsilon$-Gromov-Hausdorff close to $B(0, R)$ in $\mathbb{R}^k$. Let $R_1 = dist_{g_\infty}(q, p_\infty)$. As $(M'_i, p_i, g'_i)\to (N, p_\infty, g_\infty)$, for all large $i$, we can find points $q_i'\in B_{g'_i}(p_i, R_1+1)$ so that 
$(B_{g'_i}(q_i', rR), \frac{1}{r^2}g'_i)$ is $\epsilon$-Gromov-Hausdorff close to $(B_{g_\infty}(q, rR), \frac{1}{r^2}g_\infty)$.
Therefore, we can find $q_i\in M, d_i>0$ such that $(M_i, q_i, g_i) = (M, q_i, d_i^{-2}g)$ pointed converges to ($\mathbb{R}^k, 0, \nu$) in the measured Gromov-Haudorff sense. Moreover, $\nu$ is proportional to the standard measure on $\mathbb{R}^k$. For the last statement, one can refer to proposition $1.35$ in \cite{[CC2]} or remark 1.35 in \cite{[CC3]}.

\begin{lemma}\label{lm1}
Let $(N_i^n, p_i, \nu_i)$ be a sequence of pointed complete noncompact K\"ahler manifolds with nonnegative bisectional curvature. Here $\nu_i$ is the standard volume form on $N_i$. After certain renormalization of $\nu_i$, assume $(N_i, p_i, \nu_i)$ converges to $(N_\infty, p_\infty, \nu_\infty)$ in the measured Gromov-Hausdorff sense, where $(N_\infty, p_\infty, \nu_\infty)$ is a metric measured space which is not necessarily smooth. Let $d$ be a fixed positive number. We further assume for each $i$, there exist $k$ linearly independent holomorphic functions $g^j_i\in\mathcal{O}_d(N_i)$, where $j$ is the index from $1$ to $k$. Set $A_i= span\{g_i^j\}$. Then $A_i$ converges to a $k$ dimensional space of complex harmonic functions $A_\infty$ on $N_\infty$ with respect to the measure $\nu_\infty$. Moreover, for any $f\in A_\infty$, $f$ is of polynomial growth of order $d$ on $N_\infty$.
\end{lemma}
\begin{proof}
For each $i$, we choose a unitary frame $g_i^j$ for $A_i$ with repect to the average of the $L^2$ norm of $B(p_i, 1)\subset N_i$. That is, $$\dashint_{B(p_i, 1)}g_i^{j}\overline{g_i^s} = \delta_{js}.$$ Let $x\in B(p_i, \frac{1}{2})$. As $|g_i^j|$ is subharmonic, by the mean value inequality of Li and Schoen \cite{[LS]} and the volume comparison, $$|g_i^j(x)|^2\leq C(n)\dashint_{B(x, \frac{1}{2})}|g_i^j|^2\leq (C_1(n))^2\dashint_{B(p_i, 1)}|g_i^j|^2 = (C_1(n))^2.$$ Here $C(n), C_1(n)$ are positive constants depending only on $n$. Therefore, theorem \ref{thm0} implies that for $x\in B(p_i, r)$, $$|g_i^j(x)|\leq C_1(n)(2r)^d.$$ Here $r\geq \frac{1}{2}$. Cheng-Yau's gradient estimate \cite{[CY]} implies that $$|\nabla g_i^j|\leq C_2(n)r^{d-1}$$ in $B(p_i, r)$ for any $i$ and $r>0$. By Arzela-Ascoli theorem and lemma \ref{lemma-10}, we may assume $g^j_i$ converges to complex harmonic functions $f_j (j = 1,...,k)$ on $N_\infty$ with respect to $\nu_\infty$. Let $M_j(r)$ be the maximum of $|f_j(y)|$ for $y\in B(p_\infty, r)$. Uniform convergence and theorem \ref{thm0} imply that $\frac{M_j(r)}{r^d}$ is monotonic nonincreasing. Therefore, $f_j$ is of polynomial growth of order $d$. Moreover, since $g^j_i$ is a unitary frame and the convergence is uniform on $B(p_i, 1)$, $f_j(j=1,...,k)$ satisfies $$\dashint_{B(p_\infty, 1)}f_j\overline{f_s} = \delta_{js}.$$ Thus they are linearly independent. Define $A_\infty = span\{f_j\}$. This completes the proof of the lemma.
\end{proof}
\begin{lemma}[Ni-Tam]\label{lm2}
Let $M^n$ be a complete noncompact K\"ahler manifold with nonnegative bisectional curvature. Suppose the universal cover $\tilde{M}$ does not split as a product of two K\"ahler manifolds and there exists a nonconstant holomorphic function with polynomial growth, then $dim(\mathcal{O}_d(M))\geq cd^n$ for all sufficiently large $d$. Here $c$ is a positive constant depending only on $M$.
\end{lemma}
\begin{proof}
The proof is an application of the standard $L^2$ estimate \cite{[Ho]} and the Ni-Tam theory on plurisubharmonic functions \cite{[NT1]}. For reader's convenience, we include the details.
Assume $f\in \mathcal{O}_d(M)$ for some $d>0$ and $f$ is not constant. Let $H(x, y, t)$ be the heat kernel on $M$. Define $$g_t(x) = \int H(x, y, t)\log(|f(y)|^2+1)dy$$ where $H(x, y, t)$ is the heat kernel on $M$. Then $g_t(x)$ satisfies the heat equation $$(\frac{\partial}{\partial t}-\Delta)g_t(x) = 0$$ with initial condition $g_0(x) = \log(|f(x)|^2+1)$. It is easy to see that $g_0(x)$ is a plurisubharmonic function. Define $g(x) = g_1(x)$. Since the universal cover of $M$ does not split, by results of Ni and Tam \cite{[NT1]} (theorem 3.1, theorem 2.1,  corollary 1.4 in \cite{[NT1]}), $g_t(x)$ is strictly plurisubharmonic on $M$ for $t>0$; $(\partial\overline\partial g(x))^n >0$; $g(x)$ is of logarithmic growth: $0\leq g(x)\leq C_1\log (r+1)$ for some constant $C_1>0$. Let  $\{z_1, ..., z_n\}$ be the local coordinate near a point $p\in M$. Let $h_i = \varphi(x)z_i$, where $\varphi(x)$ is a cut-off function which has support inside the local coordinate neighborhood. Let $\theta_i=\overline\partial h_i$. Now apply theorem $3.2$ in \cite{[N2]}, with $E$ being the anti-canonical line bundle. We have
functions $\eta_i$ such that $\overline\partial\eta_i = \theta_i$ and $\eta_i(p) = 0$. Moreover, $\eta_i$ satisfies 
\begin{equation}\label{eq2}
\int_M|\eta_i|^2\exp(-Cg(x))<\infty. 
\end{equation}
It is easy to see that $f_i=\theta_i-\eta_i$ are holomorphic and form a coordinate system near $p$. Moreover, $f_i$ satisfies (\ref{eq2}). Applying the mean value inequality of \cite{[LS]}, we conclude that $f_i$ are of polynomial growth. The lower bound of $dim(\mathcal{O}_d(M))$ follows from simple dimension counting.
\end{proof}

We go back to the proof of theorem \ref{thm1}. Let $$h_d = dim(\mathcal{O}_d(M_i, q_i, g_i)) = dim(\mathcal{O}_d(M)).$$  By lemma \ref{lm2}, there exists a positive constant $c$ independent of $d$ so that
\begin{equation}\label{eq3}h_d\geq cd^n\end{equation} for all large $d$. Recall $(M_i, q_i, g_i)$ pointed converges to $(\mathbb{R}^k, 0, \nu)$ where $\nu$ is proportional to the standard measure of $\mathbb{R}^k$. Lemma \ref{lm1} says a sequence of unitary frames of $\mathcal{O}_d(M_i, q_i, g_i)$ converges to linearly independent complex harmonic functions $f_s (s = 1, ..., h_d)$ on $(\mathbb{R}^k, 0, \nu)$. Note that $f_s$ are harmonic with respect to the standard volume form of $\mathbb{R}^k$. 
\begin{remark}
In this case,  we cannot say $f_s$ are ``holomorphic", as $\mathbb{R}^k$ does not necessarily inherit a complex structure from $M_i$. For example, $k$ might be odd.
\end{remark}

There are only two cases:

Case $1$: $k\leq n$.
In this case, just observe that $$dim(\mathcal{H}_d(\mathbb{R}^k))\leq C(k)d^{k-1}\leq C(k)d^{n-1}.$$ We have a contradiction, as $$dim(\mathcal{H}_d(\mathbb{R}^k))\geq dim(span\{f_s\}) = h_d\geq cd^n$$ for sufficiently large $d$.

\medskip

Case 2: $2n-1\geq k\geq n+1$.
The argument in Case $1$ no longer works. However, we shall prove that $f_s$ are ``more" than harmonic on $\mathbb{R}^k$. 
In what follows, we will denote by $\Phi(u_1,..., u_k|....)$ any nonnegative functions depending on $u_1,..., u_k$ and some additional parameters such that when these parameters are fixed, $$\lim\limits_{u_1,..., u_k\to 0}\Phi(u_1,..., u_k|...) = 0.$$ We also let $C(n)$ be positive constants depending only on $n$. The value of $C(n)$ might change from line to line.

Recall the Cheeger-Colding theory \cite{[CC1]}. Since $(M_i, q_i, g_i)$ converges to ($\mathbb{R}^k, 0$) in the Gromov-Hausdorff sense, given any $R>1$, there exist harmonic functions $b_j (j = 1,...,k)$ on $B(q_i, 3R)$ such that \begin{equation}\label{eq4}\dashint_{B(q_i, 2R)}  \sum\limits_{j}|\nabla(\nabla b_j)|^2 +\sum\limits_{j, l}|\langle\nabla b_j, \nabla b_l\rangle - \delta_{jl}|^2\leq \Phi(\frac{1}{i}|R, n)\end{equation} and \begin{equation}\label{eq5}|\nabla b_j|\leq C(n)\end{equation} in $B(q_i, 2R)$.  Moreover, when taking a diagonal sequence with $R\to \infty$, these $b_j$ converge to the standard coordinate functions on $\mathbb{R}^k$. 

Since $M_i$ is K\"ahler, $J\nabla b_j$ satisfies (\ref{eq4}) and (\ref{eq5}). That is, we replace $\nabla b_j$ by $J\nabla b_j$. The key observation is that if $k\geq n+1$, in the average sense, $span\{\nabla b_j\}$ will have nonzero intersection with $span\{J\nabla b_j\}$ due to dimension reasons. This will give a linear complex structure for some directions of $\mathbb{R}^k$. Then we can reduce the upper bound of the dimension of $span\{f_s\}$.

\begin{definition}
We say a sequence of vector fields $s^i_l(l = 1, ...., N)$ are almost orthonormal in $B(q_i, 2R)$ if $\dashint_{B(q_i, 2R)} \sum\limits_{l}|\nabla s^i_l|^2 +\sum\limits_{m, l}|\langle s^i_l,  s^i_m\rangle - \delta_{ml}|^2\leq \Phi(\frac{1}{i}|n, R)$.
\end{definition}

\begin{claim}\label{cl-1}
Let $s^i_l(l=1,...,N)$ be a sequence of almost orthonormal vector fields in $B(q_i, 2R)$. Then $N\leq 2n$.
\end{claim}
\begin{proof}
There exists a point $x_i\in B(q_i, 2R)$ such that $\sum\limits_{m, l}|\langle s^i_l(x_i), s^i_m(x_i)\rangle - \delta_{ml}|^2\leq \Phi(\frac{1}{i}|R, n)$. Suppose $N>2n$.
If $i$ is large, there is a contradiction from linear algebra.
\end{proof}
The following is just a Schmidt orthogonalization. The argument is rather standard. However, for completeness, we include the details. 

Note that $\nabla b_j$ are almost orthogonal. 
Define \begin{equation}\label{eq6}\lambda_{j, 1} = \dashint_{B(q_i, 2R)}\langle\nabla b_j, J\nabla b_1\rangle.\end{equation}  Obviously, $|\lambda_{j, 1}|\leq C(n)$. Let \begin{equation}\label{eq7}e_1 = J\nabla b_1 - \sum\limits_{j=1}^k\lambda_{j, 1}\nabla b_j.\end{equation}

\begin{claim}\label{cl1}
$\dashint_{B(q_i, 2R)}|\langle e_1, \nabla b_j\rangle|^2 = \Phi(\frac{1}{i}|n, R)$ for $j = 1,...., k$.
\end{claim}
\begin{proof}
Define a function \begin{equation}\label{eq8}s_j(x) = \langle J\nabla b_1, \nabla b_j\rangle.\end{equation}  By Buser \cite{[Bu]}, on $M_i$, we have the Neumann-Poincare inequality \begin{equation}\label{eq9}\dashint_{B(q_i, 2R)}|s_j(x)-\overline{s_j}|^2 \leq C(n)R^2\dashint_{B(q_i, 2R)} |\nabla s_j|^2.\end{equation} Note \begin{equation}\label{eq10}\overline{s_j} = \dashint_{B(q_i, 2R)}s_j(x) = \lambda_{j, 1}.\end{equation} Also  \begin{equation}\label{eq11}|\nabla s_j|\leq |\nabla^2 b_1||\nabla b_j|+|\nabla b_1||\nabla^2 b_j|.\end{equation} From (\ref{eq4}), (\ref{eq5}) and (\ref{eq11}), 
\begin{equation}\label{eq12}
\begin{aligned}
C(n)R^2\dashint_{B(q_i, 2R)} |\nabla s_j|^2&\leq C(n)R^2\dashint_{B(q_i, 2R)}(|\nabla^2 b_1||\nabla b_j|+|\nabla b_1||\nabla^2 b_j|)^2 \\&\leq 2C(n)R^2\dashint_{B(q_i, 2R)}(|\nabla^2 b_1|^2|\nabla b_j|^2+|\nabla b_1|^2|\nabla^2 b_j|^2)\\&\leq\Phi(\frac{1}{i}|R, n). 
\end{aligned}
\end{equation} 

By (\ref{eq4})-(\ref{eq12}),  \begin{equation}\label{eq13}
\begin{aligned}
(\dashint_{B(q_i, 2R)}|\langle e_1, \nabla b_j\rangle|^2)^{\frac{1}{2}}& = (\dashint_{B(q_i, 2R)}|s_j - \lambda_{j, 1}\langle\nabla b_j, \nabla b_j\rangle -\sum\limits_{s\neq j}\lambda_{s, 1}\langle \nabla b_s, \nabla b_j\rangle|^2)^{\frac{1}{2}}\\& \leq (\dashint_{B(q_i, 2R)}|s_j-\overline{s_j}|^2)^{\frac{1}{2}}+|\lambda_{j, 1}|(\dashint_{B(q_i, 2R)}(|\nabla b_j|^2-1)^2)^{\frac{1}{2}} \\&+ \sum\limits_{l\neq j}|\lambda_{l, 1}|(\dashint_{B(q_i, 2R)}|\langle\nabla b_j, \nabla b_l\rangle|^2)^{\frac{1}{2}}\\& \leq \Phi(\frac{1}{i}|R, n).
\end{aligned}
\end{equation} 

\end{proof}

 If no subsequence of $\dashint_{B(q_{i}, 2R)}|e_1|^2$ is converging to $0$, we can rescale $e_1$ which we still call $e_1$ so that $\dashint_{B(q_{i}, 2R)}|e_1|^2= 1$. Since the rescale factor is bounded from above (independent of $i$), we still have claim \ref{cl1}. Thus $\{\nabla b_j, e_1\}$ become almost orthonormal.  
Define \begin{equation}\label{eq100}e_2 = J\nabla b_2 - \sum\limits_{j=1}^k\lambda_{j, 2}\nabla b_j-\mu_2e_1.\end{equation} Here \begin{equation}\label{eq101}\lambda_{j, 2} = \dashint_{B(q_i, 2R)}\langle\nabla b_j, J\nabla b_2\rangle; \mu_2 =\dashint_{B(q_i, 2R)}\langle e_1, J\nabla b_2\rangle.\end{equation} It is easy to check that $e_2$ satisfies claim \ref{cl1}. If no subsequence of $\dashint_{B(q_{i}, 2R)}|e_2|^2$ is converging to $0$, we can rescale it again. Then we continue to define $e_3, e_4$ and so on. Note that in general,  $e_s$ is a linear combination of $\nabla b_1,..., \nabla b_k, J\nabla b_1,...., J\nabla b_s$.
By the assumption of Case 2, $2k \geq 2(n+1)>dim_\mathbb{R}(M)$. According to claim \ref{cl-1}, this process must stop at some $e_\lambda$ for $1\leq\lambda\leq 2n-k+1< k$ for dimension reason. That is, $\dashint_{B(q_{i}, 2R)}|e_\lambda|^2$ is converging to zero for some subsequence. Passing to that subsequence, we may assume $$\dashint_{B(q_{i}, 2R)}|J\nabla b_\lambda-\sum\limits_{j=1}^kc_j\nabla b_j -\sum\limits_{j=1}^{\lambda-1}s_jJ\nabla b_j|^2\to 0$$ for some constants $c_j$ and $s_j$.
Define $b_1' = \frac{b_\lambda-\sum\limits_{j=1}^{\lambda-1}s_jb_j}{\sqrt{1+\sum\limits_{j=1}^{\lambda-1}s_j^2}}, b_2' = \frac{\sum\limits_{j=1}^kc_jb_j}{\sqrt{1+\sum\limits_{j=1}^{\lambda-1}s_j^2}}.$ By a linear transformation of $b_1,..., b_k$, we can easily extend $b_1', b_2'$ to $b_1', b_2', ...., b_k'$ which satisfy  (\ref{eq4}) and (\ref{eq5}). Observe that
 \begin{equation}\label{-6}\dashint_{B(q_i, 2R)} |J\nabla b'_1 - \nabla b'_2|^2 \leq \Phi(\frac{1}{i}|n, R).\end{equation} 
 
Then we define $e'_3, e'_4$ and so on similar as before ($e'_2$ is skipped for an obvious reason).  Assume the process stops at $e'_{\lambda'}$. Then $3\leq{\lambda}'\leq 2n-k+3$ due to dimension reason. Suppose $2n-k+3<k$, that is $k>n+1$. Then as before, we have 
\begin{equation}\label{-7}
\dashint_{B(q_{i}, 2R)}|J\nabla b'_{\lambda'}-\sum\limits_{j=1}^kc'_j\nabla b'_j -\sum\limits_{j=1}^{\lambda'-1}s'_jJ\nabla b'_j|^2\to 0\end{equation} for some constants $c'_j$ and $s'_j$. 
 By (\ref{-6}) and the almost orthogonality of ($\nabla b_1',...., \nabla b_k'$),  $\nabla b_1', \nabla b_2'$ are almost orthonormal to $\nabla b_j', J\nabla b_j'$ for $j\geq 3$. Then we have \begin{equation}\label{-8}
\dashint_{B(q_{i}, 2R)}|J\nabla b'_{\lambda'}-\sum\limits_{j=3}^kc'_j\nabla b'_j -\sum\limits_{j=3}^{\lambda'-1}s'_jJ\nabla b'_j|^2\to 0\end{equation}  
Define $b_1'' = b_1', b_2'' = b_2', b_3''= \frac{b'_{\lambda'}-\sum\limits_{j=3}^{\lambda'-1}s'_jb'_j}{\sqrt{1+\sum\limits_{j=3}^{\lambda'-1}(s'_j)^2}}, b''_4 = \frac{\sum\limits_{j=3}^kc'_jb'_j}{\sqrt{1+\sum\limits_{j=3}^{\lambda'-1}(s'_j)^2}}$.  Note $b_1'', b_2'', b_3'', b_4''$ are almost orthogonal. By a linear transformation of $b'_1,..., b'_k$, we can easily extend the functions to $b_1'', b_2'', ...., b_k''$ which satisfy (\ref{eq4}) and (\ref{eq5}).  (\ref{-8}) and (\ref{-6}) imply
\begin{equation}\label{-9}\dashint_{B(q_i, 2R)} |J\nabla b''_1 - \nabla b''_2|^2 \leq \Phi(\frac{1}{i}|n, R); \dashint_{B(q_i, 2R)} |J\nabla b''_3 - \nabla b''_4|^2 \leq \Phi(\frac{1}{i}|n, R).\end{equation} 
Continuing the process as above, after certain linear transformation, we may assume $\nabla b_j (j = 1,..., k)$ satisfy (\ref{eq4}), (\ref{eq5}) and 
\begin{equation}\label{eq14}\dashint_{B(q_i, R)} |J\nabla b_{2s-1} - \nabla b_{2s}|^2 \leq \Phi(\frac{1}{i}|n, R)\end{equation} for $1\leq s\leq k-n$. By taking $R\to\infty$ and a diagonal subsequence argument, we can define a ``partial" complex structure on the limit space $\mathbb{R}^k$: \begin{equation}\label{eq15}J\nabla b_{2s-1} = \nabla b_{2s}, J\nabla b_{2s} = -\nabla b_{2s-1}\end{equation} for 
$1\leq s\leq k-n$. Therefore we can write $\mathbb{R}^k = \mathbb{C}^{k-n}\times\mathbb{R}^{2n-k}$. Note that there is no ambiguity on the complex structure $J$ in different spaces: just check the space first.

\begin{lemma}\label{lm-1}
The functions $f_s$ are holomorphic on the $\mathbb{C}^{k-n}$ factor of $\mathbb{R}^k$.
\end{lemma}
\begin{proof}
 Let $f_s = u+\sqrt{-1}v$, where $u$ and $v$ are real harmonic functions on $\mathbb{R}^k$. By (\ref{eq15}), we just need to verify the Cauchy-Riemann equation for $f_s$ along $\nabla b_1$ and $\nabla  b_2$.  Given any point $x\in\mathbb{R}^k$, consider a smooth function $\lambda = \lambda(b_1, ...., b_k)$ supported in $B(x, 1)$.  
Take $R = |x| + 3$.  We may assume $g_i\in \mathcal{O}_d(M_i, q_i, g_i)$ converges uniformly in any compact set to $f_s$. Let $g_i = u_i+\sqrt{-1}v_i$, where $u_i$ and $v_i$ are real pluriharmonic functions on $M_i$. Moreover, \begin{equation}\label{eq16}|u_i|, |v_i|, |\nabla u_i|, |\nabla v_i| \leq C(R, d, n)\end{equation} in $B(q_i, R)$. By Cauchy-Riemann equation, $\langle\nabla u_i, \nabla b_1\rangle = \langle\nabla v_i, J\nabla b_1\rangle$. Note that for sufficiently large $i$, $\lambda(b_1, ..., b_k)$ is supported in $B(q_i, R)\subset M_i$ (here for $x\in B(q_i, R)$, $\lambda(b_1,..., b_k)$ is defined by $\lambda(b_1(x), ...., b_k(x))$ ). (\ref{eq14}) and (\ref{eq16}) imply
\begin{equation}\label{eq17}
|\dashint_{B(q_i, R)}\lambda(b_1,..., b_k)\langle\nabla v_i, J\nabla b_1\rangle -\dashint_{B(q_i, R)}\lambda(b_1,..., b_k)\langle\nabla v_i, \nabla b_2\rangle| \leq \Phi(\frac{1}{i}|n)\end{equation} Here $R$ and $\lambda$ are already fixed. As $b_i$ are harmonic, (\ref{eq4}), (\ref{eq5}) and (\ref{eq16}) imply
 \begin{equation}\label{eq18}
 \begin{aligned}\dashint_{B(q_i, R)}\lambda(b_1,..., b_k)\langle\nabla u_i, \nabla b_1\rangle& =-\dashint_{B(q_i, R)}u_i\langle\nabla (\lambda(b_1,...., b_k)), \nabla b_1\rangle \\&=-\dashint_{B(q_i, R)}u_i\sum\limits_{j=1}^k\frac{\partial\lambda}{\partial b_j}\langle\nabla b_j, \nabla b_1\rangle\\&\to -\dashint_{B(0, R)}u\frac{\partial \lambda}{\partial b_1}\\&=\dashint_{B(0, R)}\lambda\langle\nabla u, \nabla b_1\rangle.
 \end{aligned}
 \end{equation} Similarly
 \begin{equation}\label{eq19} \dashint_{B(q_i, R)}\lambda(b_1,..., b_k)\langle\nabla v_i, \nabla b_2\rangle \to \dashint_{B(0, R)}\lambda\langle\nabla v, \nabla b_2\rangle.
 \end{equation}
 
 By (\ref{eq17})-(\ref{eq19}) and that $\lambda$ is supported in $B(x, 1)$, we find $$\dashint_{B(x, 1)}\lambda\langle\nabla u, \nabla b_1\rangle = \dashint_{B(x, 1)}\lambda\langle\nabla v, \nabla b_2\rangle.$$ Similarly $$\dashint_{B(x, 1)}\lambda\langle\nabla v, \nabla b_1\rangle = \dashint_{B(x, 1)}-\lambda\langle\nabla u, \nabla b_2\rangle.$$ Since $\lambda$ is arbitrary, $$\langle\nabla u, \nabla b_1\rangle = \langle\nabla v, \nabla b_2\rangle, \langle\nabla v, \nabla b_1\rangle = -\langle\nabla u, \nabla b_2\rangle$$ at $x$. This concludes the proof of the lemma.
  \end{proof}
  
Let $H(d, k)$ be the space of complex harmonic functions in $\mathbb{R}^k$ with polynomial growth rate $d$. We identify $\mathbb{R}^k = \mathbb{C}^{k-n}\times\mathbb{R}^{2n-k}$.
 Let $E(d, n, k)$ be the subspace of $H(d, k)$ so that the restriction to the $\mathbb{C}^{k-n}$ factor is holomorphic. Then $$E(d, n, k)\subset span\{fg\}$$ where $f\in \mathcal{O}_d(\mathbb{C}^{k-n})$ and $g\in \mathcal{H}_d(\mathbb{R}^{2n-k})$. Therefore \begin{equation}\begin{aligned}
 dim(E(d, k, n))&\leq dim(\mathcal{O}_d(\mathbb{C}^{k-n}))dim(\mathcal{H}_d(\mathbb{R}^{2n-k}))\\&\leq C(n, k)d^{k-n}d^{2n-k-1} \\&= C(n, k)d^{n-1}.\end{aligned}\end{equation}
 Recall $f_s(s= 1,...., h_d)$ are linearly independent. Moreover $f_s\in E(d, n, k)$ by lemma \ref{lm-1}. Therefore $$h_d\leq C(n, k)d^{n-1}.$$ This contradicts (\ref{eq3}).
The proof of theorem \ref{thm1} is complete.

\section{\bf{Proof of the corollaries}}

\emph{Proof of corollary \ref{cor1}:}
This directly follows from the proof of theorem \ref{thm1}. Note that in theorem \ref{thm1}, the condition $\tilde{M}$ does not split is only used to show lemma \ref{lm2}. Note that throughout the proof of corollary \ref{cor1}, we only assume the Ricci curvature and the holomorphic sectional curvature are nonnegative. This is slightly weaker than the nonnegativity of the bisectional curvature.

\medskip

\emph{Proof of corollary \ref{cor2}:}
By the assumption of corollary \ref{cor2} and theorem $1.2$ in \cite{[NT2]}, we can solve the Poincare-Lelong equation $\sqrt{-1}\partial\overline\partial u = Ric$ where $Ric$ is the Ricci form of $M$. Theorem $1.2$ in \cite{[NT2]} also implies $u$ is of logarithmic growth. As the Ricci curvature is positive at one point $p\in M$, $u$ is plurisubharmonic and strictly plurisubharmonic at one point. Now we can replace the function $g(x)$ by $u(x)$ in lemma \ref{lm2} to deduce that $dim(\mathcal{O}_d(M))\geq cd^n$ for all sufficiently large $d$.

\section{\bf{Sharp dimension estimates revisited}}
In this section, we discuss Ni's sharp dimension estimates (theorem \ref{thm-9}) from the point of view of theorem \ref{thm1}. We will not include the rigidity part here.  Under the assumption of theorem \ref{thm-9} (without the maximal volume growth), let $(M_i, g_i, p_i) = (M, \frac{1}{r_i}g, p)$ where $r_i$ is a positive sequence converging to $0$. Then it is easy to see $(M_i, g_i, p_i)$ converges to $\mathbb{C}^n$. If theorem \ref{thm-9} is not true for some $d$, there is a contradiction with lemma \ref{lm1}.

\end{document}